\documentclass[reqno]{amsart} 
\usepackage{amsmath}
\usepackage{paralist}
\usepackage[misc]{ifsym}
\usepackage{epsfig} 
\usepackage{epstopdf} 
\usepackage[colorlinks=true]{hyperref}
\hypersetup{urlcolor=blue, citecolor=red}
\allowdisplaybreaks
\DeclareMathOperator{\dist}{dist}
\DeclareMathOperator{\sech}{sech}
\DeclareMathOperator{\arcsech}{arcsech}

\newtheorem{theorem}{Theorem}[section]
\newtheorem{corollary}[theorem]{Corollary}
\newtheorem{lemma}[theorem]{Lemma}

\title[Domain wall system]
{Variational approach for the singular perturbation domain wall  system} 

\author[Javier Monreal and Micha{\l} Kowalczyk]{}


\thanks{The first author was supported by FONDECYT grant 1210405 and ANID doctoral fund 21242106. The second author was supported by   research grants FONDECYT 1210405 and ANID projects ACE210010 and FB210005.}

\begin{document}
\maketitle

\centerline{\scshape
Javier Monreal $^{{\href{mailto:jmonreal@dim.uchile.cl}{\textrm{\Letter}}}}$
and Micha{\l} Kowalczyk $^{{\href{mailto:kowalczy@dim.uchile.cl}{\textrm{\Letter}}}}$}

\medskip

{\footnotesize
 \centerline{Departamento de Ingeniería Matemática}
 \centerline{and Centro de Modelamiento Matemático
(UMI 2807 CNRS)}
\centerline{Universidad de Chile, Casilla 170 Correo 3, Santiago, Chile}
} 

\medskip


\bigskip



\begin{abstract}
In this article we study a coupled system of differential equations with Allen-Cahn type non-linearity. Motivated by physical phenomena one of the unknowns in the system is accompanied by a singular perturbation parameter $\varepsilon^2$. By employing variational techniques, we establish the existence of solutions for all values of $\varepsilon$ and get results on their qualitative properties, including regularity. Additionally, we analyse the behaviour of solutions as $\varepsilon \to 0$, demonstrating their pointwise convergence to the solution of the problem for $\varepsilon=0$. We establish the uniqueness of this solution modulo translations. Additionally, in the final section, through an appropriate change of scale, we relate this problem and the second  Painlevé  equation.
\end{abstract}


\section{Introduction}

Domain walls are usual defects with a well-known origin in the theory of media with a vectorial local-order parameter, such as magnetics \cite{DynTopMag}, ferroelectrics \cite{Ferroelectric}, and liquid crystals \cite{ThePhysicsOfLiquidCrystals}. Significant examples of the appearance of this defect are also found in the study of non-linear systems in the modelling of reaction-diffusion problems. For instance, it can be observed that in patterns on the surface of thermal convection layers  \cite{Cross}, \cite{Pismen} the domain walls appear as the transition between different sets of stripes. The usual model to address this problem involves the use of amplitude equations, resulting in a system of coupled Ginzburg-Landau (GL) equations:
\[
\begin{aligned}
 \frac{\partial u}{\partial t}&= D_1\frac{\partial^2u}{\partial x^2}+u(1-u^2-\mu v^2),\\
 \frac{\partial v}{\partial t}&= D_2\frac{\partial^2v}{\partial x^2}+v(1-v^2-\mu u^2).
\end{aligned}
\]
Stationary domain walls satisfy 
\begin{equation}
\label{eq:dw stationary}
\begin{aligned}
D_1u''+u(1-u^2-\mu v^2)&=0,\\
 D_2 v''+v(1-v^2-\mu u^2)&=0,
\end{aligned}
\end{equation}
where we seek for non-negative solutions, with heteroclinic boundary at infinity:
\begin{equation}
\label{cond infty}
\begin{aligned}
 u(x)\to0, \; v(x)\to 1,\; \text{ as } x\to-\infty,\\
 u(x)\to1, \; v(x)\to 0,\; \text{ as } x\to +\infty.
\end{aligned}
\end{equation}

 The superposition of the rolls with their corresponding amplitudes (obtained as solutions of the GL system) forms the domain wall.

Similar systems of coupled Non-linear Schrödinger (NLS) equations and Gross-Pitaevskii (GP) equations describe, respectively, the co-propagation of electromagnetic waves with orthogonal polarizations in non-linear optical fibres and binary mixtures of Bose-Einstein condensates in cigar-shaped traps \cite{Jaksch2003}. Consider the system of coupled GP equations,
\[
\begin{aligned}
 i\frac{\partial \psi_1}{\partial t}=- \frac{\partial^2\psi_1}{\partial x^2}+\psi_1(g_{11}|\psi_1|^2+g_{12}|\psi_2|^2),\\
 i\frac{\partial \psi_2}{\partial t}=- \frac{\partial^2 \psi_2}{\partial x^2}+\psi_2(g_{12}|\psi_2|^2+g_{22}|\psi_1|^2).
\end{aligned}
\]
Standing wave solutions, $\psi_1=e^{\,-it}u, \psi_2=e^{\,-it}v$ with  $g_{11}=1=g_{22}, g_{12}=\mu$ coincide with the problem (\ref{eq:dw stationary}) when conditions at infinity for the envelope are set. While GL systems are dissipative, NLS and GP systems are conservative, hence, stability results for the solutions can differ. Using variational techniques, there has been many recent and significant results on the existence and stability of the standing waves for the Gross-Pitaevskii systems \cite{Stan}, \cite{Contreras2018}, \cite{Contreras2022}.

On the other hand, by using classical techniques, interesting results have been obtained concerning solutions to the  Thomas-Fermi (TF) approximation, where $D_2=0$  in the equation (\ref{eq:dw stationary}). The existence and characterization of the solution in this case have been studied in \cite{Malomed}.

Building on this, we present  here results concerning the  relation between the two mentioned problems, namely system  (\ref{eq:dw stationary}) and its associated TF limit approximation. To fix attention we set $D_1=1$ and consider $D_2=\varepsilon^2$ to be a small parameter. This system has actually physical interpretation: it correspond to the situation when the angle between the stripe patterns  is close to $\pi/2$ \cite{Pismen}. We note that (\ref{eq:dw stationary}) is variational and so is its formal Thomas-Fermi limit. In this paper we take advantage of this to study the relation between the minimizers of both problems under conditions (\ref{cond infty}) in the limit of the singular parameter. 
 For the TF problem, we prove the uniqueness of the minimizers. For the stationary GL system we show  $C^{1,\alpha}_{loc}$ convergence for the first component and point-wise  convergence for the second component  to the unique minimizer of the TF limit as the singular parameter approaches zero. This result is natural in the view of the fact that the TF limit is only H\"older continuous and presents what is known as the corner layer singularity. 
 This phenomenon has been shown to occur in related ordinary differential equation problems \cite{Alikakos2, Alikakos3, shadowkink}. Following the ideas of these works we establish here the connection between our singular perturbation problem and the Painlevé II equation, when the appropriate change of variables is applied to the solutions.

The paper is organized as follows. In Section \ref{Settingof}, we rigorously present the problem, including the system of equations, the corresponding energy functional, and the energy space in which we will work.  In Section \ref{existencesection}, we prove the existence and provide characterization of the minimizers  of the GL coupled system of equations when the perturbation parameter is considered to be fixed. We show uniform estimates in $L^\infty$ and monotonicity of the minimizers. In Section \ref{SingularLimitProblem}, we present the well-known solutions for the TF approximation problem \cite{Malomed}, and prove uniqueness in $H^1_{loc}\times C$ when the corresponding conditions at infinity are satisfied. Section \ref{LimitBehaviour} is devoted to the study of the behavior of the solutions when the perturbation parameter tends to zero. We find that there is convergence to the unique solution (up to translations) of the TF limit approximation problem. For the proof, we employ compactness techniques, achieved through the correct use of energy estimates. Finally, in Section \ref{PainleveProblem}, using the estimates from the previous section, we show that by applying the appropriate change of variables the limit equation becomes the Painleve II equation, similar to the results obtained in \cite{Alikakos2, Alikakos3, shadowkink}.

\section{Setting of the system}
\label{Settingof}
We are interested in working with the following coupled system of ordinary differential equations (c.f (\ref{eq:dw stationary})):
\begin{equation}
\label{system}
\begin{cases}
u''+u(1-u^2-\mu v^2)=0,\\
\varepsilon^2 v'' +v(1-v^2-\mu u^2)=0,
\end{cases}
 \end{equation}
 where $\varepsilon>0$ is a small parameter. 
Motivations for the study of this system are well established \cite{Cross}, \cite{Pismen}.
We start by setting up the problem of existence, proposing a variational approach considering the following energy functional:
\begin{equation}
\label{energy}
E_\varepsilon(u,v)=\int_{-\infty}^{\infty} \frac{(u')^2}{2}+\varepsilon^2 \frac{(v')^2}{2}+\frac{W(u,v)}{2},
\end{equation}
where $W$ is defined by
\begin{equation}
W(u,v):=\frac12 (u^2+v^2-1)^2+(\mu-1)\frac{u^2 v^2}{2}.
\end{equation}
We will minimize this functional in  the following set of functions:
\[
X:=\{(u,v)\in H^1_{loc}\times H^1_{loc} | (u,v)\to (1,0),  x\to +\infty, \text{ and } (u,v)\to (0,1),  x\to-\infty \}.
\]
\section{General existence results}
\label{existencesection}
For the following proofs in this section we will denote
\[
E_\varepsilon(u,v)=\int e_\varepsilon(U(x))dx,
\]
where $U(x)=(u(x),v(x))$, and
\[
e_\varepsilon(U)=\frac{(u')^2}{2}+\varepsilon^2\frac{(v')^2}{2}+\frac{W(u,v)}{2}.
\]
We gather the properties of the potential $W$.
\begin{lemma}
We have 
\begin{itemize}
\item[(W1)] $W\in C^\infty (\mathbb{R}^2;\mathbb{R})$.
\item[(W2)] $W(u,v)\geq 0$ for all $(u,v)\in \mathbb{R}^2$, and $W(u,v)=0$ if and only if $(u,v)=(1,0)$ or $(u,v)=(0,1)$.
\item[(W3)]$(1,0)$ and $(0,1)$ are non-degenerate global minima of $W$.
\item[(W4)] There exist constants $R_0,c_0 >0$ such that
\[
\nabla W(u,v)\cdot (u,v) \geq c_0 |(u,v)|^2 \text{ for all } (u,v)\in \mathbb{R}^2_+\text{ with } |(u,v)|\geq R_0,
\]
and as a consequence,
\[
\int_{(-R,R)}|U|^2 dx\leq \int_{(-R,R)}W(U)+2R R_0,
\]
for any value of $R>0$.
\end{itemize}
\end{lemma}
The results of this section are adapted to our context  from \cite{Stan}, as indicated below, and are included here for completeness.  

\begin{lemma}[\cite{Stan}]
\label{firstlemma}
 Let U be in $H^1_{loc}(\mathbb{R},\mathbb{R}^2)$, such that the energy defined above has finite value. Then $\lim_{x\to \pm \infty}W(U(x))=0$.
\begin{proof}
Denote $a=(0,1)$ and $b=(1,0)$. 
Since $a,b$ are non-degenerate global minimizers of $W$, there exists constants $C, \lambda, \Lambda,\delta>0$ such that
\begin{equation}
\label{est:star}
\frac{1}{C}|U-a| \leq \lambda|U-a| \leq \sqrt{W(U)} \leq \Lambda |U-a| \leq C |U-a| ,
\end{equation}
for any $U\in \mathbb{R}^2$ if $|U-a|\leq \delta$ for sufficiently small $\delta$. The same holds for $b$ instead of $a$.

Suppose $W(U(x))$ does not approach zero at infinity; then there must exist a sequence $(x_n)_n$ with $x_n\to \infty$ such that $W(U(x_n))>\eta$ for some $\eta>0$. From the previous observations, we deduce that we can choose a constant $C>0$ such that
\[
\min\{|U(x_n)-a|,|U(x_n)-b| \}\geq C^{-1}\sqrt{\eta}.
\]
Let us denote $\delta_0:=C^{-1}\sqrt{\eta}$.

Since the integral of $W(U(x))$ is finite, there also must exist a sequence $(t_n)_n$, tending to infinity, for which $W(U(t_n))\to 0$. Without loss of generality, we can assume $x_n$ and $t_n$ are interlaced in such a way that $t_n>x_n$ is the smallest value for which\[
|U(t_n)-a|=\frac12 \delta_0, \text{ or } |U(t_n)-b|=\frac12 \delta_0,\, \forall n\in \mathbb{N},
\]
where at least one of the conditions must occur infinitely often. Let us assume  $|U(t_n)-a|=\frac12 \delta_0$ happens infinitely many times and consider that subsequence (also interlaced with $x_n$). Notice that
\[
2 \sqrt{W(U(x))}|U'(x)|\leq W(U(x))+|U'(x)|^2,
\]
by integrating,
\[
\frac12 \int_{x_n}^{t_n} W(U(x))+|U'(x)|^2 dx \geq  \int_{x_n}^{t_n} \sqrt{W(U(x))}|U'|dx.
\]
With the above, when $\varepsilon>1$ we have
\[
\int_{x_n}^{t_n} e_\varepsilon(U)dx\geq  \int_{x_n}^{t_n} \sqrt{W(U(x))}|U_n'|dx.
\]
At the same time, when $\varepsilon \leq 1$, we have $\|U'\|\leq \varepsilon^{-2} \left((u')^2 +\varepsilon^2 (v')^2 \right)$, hence
\[
\varepsilon^{-2}\int_{x_n}^{t_n} e_\varepsilon(U)dx\geq \frac{1}{2}\int_{x_n}^{t_n} W(U(x))+|U'(x)|^2 dx \geq  \int_{x_n}^{t_n} \sqrt{W(U(x))}|U'|dx.
\]
Considering all values of $\varepsilon>0$ we obtain
\[
\max(1,\varepsilon^{-2}) \int_{x_n}^{t_n} e_\varepsilon(U)dx\geq \int_{x_n}^{t_n} \sqrt{W(U(x))}|U'|dx=\int_{\sigma_n} \sqrt{W}ds,
\]
where $\sigma_n$ is the path in $\mathbb{R}^2$ parametrized by $U(x)$ for $x\in (t_n,x_n)$.
We observe that the arclenght of $\sigma_n$ is at least $\delta_0/2$, therefore
\begin{align*}
\max(1,\varepsilon^{-2}) \int_{x_n}^{t_n} e_\varepsilon(U(x))dx &\geq \int_{\sigma_n} \sqrt{W}ds \\ &\geq \int_{\sigma_n} C^{-1}\cdot |U(x)-a| ds \\
&\geq C^{-1}\cdot |U(t_n)-a| \cdot \int_{\sigma_n} ds\\
&=C^{-1}\cdot \left(\frac{\delta_0}{2}\right)^2.
\end{align*}
The third inequality comes from the fact that for all $x\in (x_n,t_n)$, the distance between $U(x)$ and $a$ must be greater than $\frac{1}{2}\delta_0$, this is because $t_n$ was chosen as the first  $x$ greater than $x_n$ taking this value, where $|U(x_n)-a|>\delta_0$. From the previous observations, the integral of $e_\varepsilon(U(x))$ cannot be finite, leading to a contradiction.
\end{proof}
\end{lemma}
Set $m:=\inf_{U\in X} E(U)$.

\begin{lemma}[\cite{alama1997}]
\label{secondlemma}
 Let $U(x)=(u(x),v(x))\in H^1_{loc}([L_1,L_2];\mathbb{R}^2)$, with $|U(L_1)-(0,1)|< \delta$ and $|U(L_2)-(1,0)|<\delta$, where $\delta>0$. Then, there exists a constant $C_1>0$ such that
 \[
\int_{[L_1,L_2]}e_\varepsilon(U(x))dx\geq m-C_1\left[|U(L_1)-(0,1)|^2 +|U(L_2)-(1,0)|^2 \right].
 \]
\begin{proof}
The proof is straightforward, based on constructing a function whose values coincide with $U$ in $[L_1,L_2]$, then using linear interpolation to reach $a=(0,1)$ and $b=(1,0)$ respectively on the left and right-hand side of the interval, interpolating over a unit length interval. This can be interpreted as computing the energy of the function defined by:
\[
\tilde{U}(x):=\begin{cases} U(x) &\text{ for } x\in [L_1,L_2]\\
(L_1-x) a+ (x-L_1+1)U(L_1) &\text{ for } x\in [L_1-1,L_1]\\
(L_2+1-x) U(L_2)+(x-L_2) b &\text{ for } x\in [L_2,L_2+1]\\
a&\text{ for } x\in (-\infty,L_1-1]\\
b &\text{ for } x\in [L_2+1,+\infty).\end{cases}
\]
Since  $m:=\inf_{U\in X} E(U)$, then
\[
 E(\tilde{U})\geq m,
\]
and thus
\[
\int_{[L_1,L_2]}e_\varepsilon(U(x))dx + \int_{[L_1-1,L_1]}e_\varepsilon(\tilde U(x))dx+\int_{[L_2,L_2+1]}e_\varepsilon(\tilde U(x))dx\geq m.
\]
Now, we make the following observation:
\[
\int_{[L_1-1,L_1]} \frac{(\tilde{u}')^2}{2}+\varepsilon^2\frac{(\tilde{v}')^2}{2} dx\leq C\cdot |U(L_1)-a|^2,
\]
and
\[
\int_{[L_2,L_2+1]} \frac{(\tilde{u}')^2}{2}+\varepsilon^2\frac{(\tilde{v}')^2}{2} dx\leq C\cdot |U(L_2)-b|^2,
\]
for some non-negative constant $C$. On the other hand, using the fact that for sufficiently small $|U-a|$, the inequality $W(U(x))\leq C |U(x)-a|^2$ holds by (\ref{est:star}) we obtain
\[
\int_{[L_1-1,L_1]} W(\tilde{U}(x))dx \leq C |U(L_1)-a|^2,
\]
and similarly
\[
\int_{[L_2,L_2+1]} W(\tilde{U}(x))dx \leq C |U(L_2)-b|^2.
\]
We conclude that
\[
\int_{[L_1,L_2]}e_\varepsilon(U(x))dx \geq m-C \left(|U(L_1)-a|^2+|U(L_2)-b|^2 \right)
\]
if $|U(L_1)-b|$ and $|U(L_2)-a|$ are small enough, where the constant depends only on $W,$ and $\delta$.
\end{proof}
\end{lemma}

Using the above we prove existence of solutions of the problem \eqref{system} for any real value of $\varepsilon$.
\begin{theorem}[\cite{Stan}]
\label{existenceproof}
For all values of $\varepsilon>0$, there exists a minimizer of $E_\varepsilon(\cdot,\cdot)$ defined by \eqref{energy} in the set $X$.

\begin{proof}
Let $(\tilde{u}_n,\tilde{v}_n)_n$ be a minimizing sequence of the functional $E_\varepsilon(\cdot,\cdot)$ over $X$, and $\tau_n$ the smallest value for which $\tilde{u}(\tau_n)=\tilde{v}(\tau_n)$. Define $(u_n(x),v_n(x)):=(\tilde{u}_n(x+\tau_n),\tilde{v}(x+\tau_n))$. Clearly, $\| u_n' \|_{L^2} $ and $\|v_n'\|_{L^2}$ are uniformly bounded. By (W4), $u_n,v_n\in L^2(-R,R)$ for all $n$, and they also are uniformly bounded (in the $L^2(-R,R)$ norm) for any fixed $R>0$. Thus, $u_n,v_n \in H^1(-R,R)$ are uniformly bounded for all $n$.

Then there exists an a.e. pointwise convergent subsequence also denoted by $(u_n,v_n)$, converging to $(u,v)$. Moreover, $(u_n,v_n)$ converges uniformly on every compact subset of $\mathbb{R}$, and weakly in $H^1(-R,R)$, and so there exists a point-wise convergent subsequence on $\mathbb{R}$. By a direct application of Fatou's Lemma and the weakly lower semi-continuity of the norm, we conclude that this function is a minimizer.

Now, we have to see that indeed $(u,v)\in X$; let us check that it takes the desired values at $\pm \infty$. By Lemma \ref{firstlemma}, we know that $\lim_{x\to \pm \infty} W(u(x),v(x))=0$. Assume, for a contradiction, that $(u(x),v(x))\to (0,1)=a$ as $x\to\infty$. We claim that using Lemma \ref{secondlemma} and (W2)  this gives a contradiction with the fact that the chosen sequence is a minimizing sequence, creating a non-zero gap between the energy of the sequence and the minimizer.

Indeed, if $U(x)\to a$ as $x\to \infty$ then, for any $\eta>0$, there exists $L_1>0$ with $|U(L_1)-a|<\eta$. Since $U_n\to U$ locally uniformly, there exists sufficiently large $n$ such that $|U_n(L_1)-a|<2 \eta$, and some other $L_2>L_1$ such that $|U_n(L_2)-b|\leq \eta$. By Lemma \ref{secondlemma}, taking small enough $\eta$, we obtain
\begin{align*}
\int_{[L_1,L_2]}e_\varepsilon(U_n(x))dx &\geq m-C \left(|U_n(L_2)-b|^2+|U_n(L_1)-a|^2 \right),\\
&\geq m-5 C_1 \eta^2.
\end{align*}
Consider now the diagonal line in the phase space of $(u,v)$, denoted by $\Delta:=\{(u,v)\in \mathbb{R}^2_+|u=v\}$. We already know that $U_n(0)\in \Delta$. Fix now $\delta>0$ small enough such that $\dist(a,\Delta),\dist(b,\Delta)>2 \delta$, and let $x_n<0$ be the largest negative value for which $|U_n(x_n)-a|=\delta$. By hypothesis (W2) and (W4), there exists $w_0>0$ with $\sqrt{W(V)}\geq w_0$ for all $V\in \mathbb{R}_+^2$ with $\dist(a,V),\dist(b,V) \geq 2 \delta$. Consider now $D=\dist(\Delta,B_\delta(a))$, which is clearly positive. Just as in Lemma \ref{firstlemma}, we notice that \[
2 \sqrt{W(U_n(x))}|U_n'(x)|\leq W(U_n(x))+|U_n'(x)|^2,
\]
by integrating,
\[
\frac12 \int_{x_n}^0 W(U_n(x))+|U_n'(x)|^2 dx \geq  \int_{x_n}^0 \sqrt{W(U_n(x))}|U_n'|dx.
\]
For all values of $\varepsilon>0$ we obtain
\begin{align*}
\max(1,\varepsilon^{-2}) \int_{x_n}^0 e_\varepsilon(U_n)dx&\geq \int_{x_n}^0 \sqrt{W(U_n(x))}|U_n'|dx\\&=\int_{\{U_n(x): x_n\leq x\leq 0\}} \sqrt{W}ds \\&\geq w_0 D,
\end{align*}
where the last inequality comes from the fact that $U_n(0)\in \Delta$, and $U_n(x_n)\in \partial B_\delta(a)$.
Thus we have shown that for any $n$:
\[
\max(1,\varepsilon^{-2}) \int_{x_n}^0 e_\varepsilon(U_n(x))dx \geq \int_{x_n}^0 \sqrt{W(U_n(x))}dx \geq w_0 D.
\]
Combining the two previous parts, we obtain:
\begin{align*}
\int_{-\infty}^\infty e_\varepsilon(U_n)dx &\geq \left[\int_{x_n}^0+\int_{L_1}^{L_2} \right]e_\varepsilon(U_n)dx\\ & \geq m+\frac{w_0 D}{\max(1,\varepsilon^{-2})}-5C_1\eta^2.
\end{align*}
By taking $\eta$ small enough, the previous inequality implies the existence of a non-zero gap between the energy of the $(U_n)_n$ family and the minimizer, contradicting the fact that $(U_n)_n$ is a minimizing sequence. We conclude that $U(x)\to (1,0)=b$ as $x\to +\infty$, and by the same argument that $U(x)\to (0,1)=a$ as $x\to -\infty$.
\end{proof}
\end{theorem}
It is also possible for us to find a $L^\infty$ estimate for $u$ and $v$ that allow us to write the following result:
\begin{corollary}[\cite{Stan}]
The solutions $(u,v)$ for \eqref{system} satisfying $({u},{v})\to (1,0)$ as $x\to +\infty$, and $({u},{v})\to (0,1)$ as $x\to -\infty$, are taking values between 0 and 1.
\begin{proof}
We know from the previous theorem that the exists a solution for the problem. Let us see that necessarily the values are taken between $0$ and $1$. We define the function $\varphi:=\left(u_\varepsilon^2+ v_\varepsilon^2-1 \right)$ and observe that
\[
-\frac12 \varphi''+(u_\varepsilon^2+\frac{v_\varepsilon^2}{\varepsilon^2})\varphi=-(u_\varepsilon')^2-(v_\varepsilon')^2-u_\varepsilon(u_\varepsilon''+u_\varepsilon(1-u_\varepsilon^2-v_\varepsilon^2))-v_\varepsilon( v''_\varepsilon+\frac{v_\varepsilon}{\varepsilon^2}(1-u_\varepsilon^2-v_\varepsilon^2))
\]
\[
=-(u_\varepsilon')^2-(v_\varepsilon')^2-(\mu-1 )v_\varepsilon^2 u_\varepsilon^2-\frac{\mu-1}{\varepsilon^2} v_\varepsilon^2 u_\varepsilon^2\leq 0.
\]
Noticing that the function $\varphi_+(x):=\max(0,\varphi(x))$ tends to $0$ at $\pm \infty$  we can multiply the previous equation by $\varphi_+$ and integrate to obtain
\[
\int_\mathbb{R} \left(\frac12  (\varphi'_+)^2+(u_\varepsilon^2+\frac{v_\varepsilon^2}{\varepsilon^2})\varphi_+^2 \right)\leq 0,
\]
so we conclude that $\varphi(x) \leq 0$ for all $x\in \mathbb{R}$, and thus $u_\varepsilon^2(x)+v_\varepsilon^2(x)\leq 1$.
Also, given that $(u,v)\in X$ implies $(|u|,|v|)\in X$, and $E(u,v)=E(|u|,|v|)$ the result is obtained.
\end{proof}
\end{corollary}
About the monotonicity of the solutions $u_\varepsilon$ and $v_\varepsilon$ we recall the following known result
\begin{theorem}[\cite{Stan}]
Assume $(W1)$-$(W4)$, and let $(u_\varepsilon,v_\varepsilon)$ be any energy minimizing solution of \eqref{energy} in $X$, then $u_\varepsilon'(x)>0$ and $v_\varepsilon'(x)<0$ for all $x\in \mathbb{R}$.
\end{theorem}

\section{Singular limit problem}
\label{SingularLimitProblem}
In this section, we explore the limit problem when $\varepsilon=0$, defining functions and an energy functional which will play an important role in the following sections.

For $\varepsilon=0$, the equation \eqref{system} admits a solution $(u_0,v_0)$ where
\begin{equation}
\label{u_epsilon}
u_0(x)=\begin{cases}
\sqrt{\frac{2}{\mu+1}}\sech(\sqrt{\mu-1}(x-x_1)) \text{ when } x\leq z,\\
\tanh\left(\frac{x-x_0}{\sqrt{2},}\right) \text{ when } x\geq z,
\end{cases}
\end{equation}
\begin{equation}
\label{v_epsilon}
v_0(x)= \begin{cases}
\sqrt{1-\frac{2\mu}{\mu+1} \sech^2(\sqrt{\mu-1}(x-x_1))} \text{ when } x\leq z,\\
0 \text{ when } x\geq z,
\end{cases}
\end{equation}
where $z$ is a free parameter, $x_0$ solves
\[ \tanh\left(\frac{z-x_0}{\sqrt{2}}\right)=\sqrt{\frac{2}{\mu+1}}\sech(\sqrt{\mu-1}(z-x_1)), \]
 and $x_1$ satisfies
\[
x_1 \in z-\frac{1}{\sqrt{\mu-1}}\arcsech \left(\sqrt{\frac{\mu+1}{2\mu}}\right).
\]
These functions together form the minimizer (on the set of functions with the respective conditions at infinity) of the functional
\[
\tilde{E}(u,v):=\int_{-\infty}^{\infty} \frac{(u')^2}{2}+\frac12 (u^2+v^2-1)^2+\frac{(\mu-1)}{2}u^2 v^2.
\]
Next, we present the results corresponding to this section.
\begin{lemma}
\label{u0v0teo}
For $\varepsilon=0$, equation \eqref{system} admits a solution $(u_0,v_0)$, where $u_0$ is a Lipchitz continuous function and $v_0$ is a Hölder continuous function with exponent $\frac{1}{2}$.

Moreover, these functions are the unique up to  translations solutions in $H_{loc}^1\times C$, that fulfil conditions at infinity $(u,v)\to (1,0)$ as $x\to \infty$, $(u,v)\to(0,1)$ as $x\to -\infty$, and minimize
\[
\tilde{E}(u,v):=\int_{-\infty}^{\infty} \frac{(u')^2}{2}+\frac12 (u^2+v^2-1)^2+\frac{(\mu-1)}{2}u^2 v^2.
\]

\begin{proof}
The fact that $(u_0,v_0)$ is a solution of the problem is straightforward calculation which we omit here.
Suppose there exists a minimizer $(u,v)$ of $\tilde{E}$. This implies that $(u,v)$ solves the associated Euler-Lagrange equations:
\begin{equation}
\begin{cases}
u''+&u(1-u^2-\mu v^2)=0\\
&v(1-v^2-\mu u^2)=0.
\end{cases}
\end{equation}
We will establish the existence of $x^{**}$ such that $v(x)=0$, for all $x>x^{**}$ by exploiting the minimization property of $(u,v)$. Similarly, we will find $x^*$ such that $(1-v^2-\mu u^2)=0$ for all $x<x^*$, which follows from the continuity of $v$ and the necessity of $v\to 1$ as $x\to -\infty$.

We take
\[
x^{**}=\inf \{x| u(x)\geq \frac{1}{\sqrt{\mu}} \},
\]
and
\[
x^*=\sup \{x<x^{**}| v^2(s)=1-\mu u^2(s),\, \forall s<x  \}.
\]
Let us write 
\[
\tilde{E}(u,v) =\int \frac{(u')^2}{2}+\frac{(1-u^2)^2}{4}+\frac{v^4}{4}+\frac{v^2}{2}(\mu u^2-1),
\]
If $u(x)\geq \frac{1}{\sqrt{\mu}}$ then  for $x>x^{**}$  the energy density is minimized by $v(x)=0$ hence $u''+u(1-u^2)=0$ and so $u=\tanh\left(\frac{x-C_1}{\sqrt{2}} \right)$ for $x>x^{**}$. At the same time if $x<x^*$, then $u''+(\mu-1)u(u^2(\mu+1)-1)=0$, and by assuming $u'\to 0, u\to 0$ as $x\to -\infty$ we get that the only solution is $u=\sqrt{\frac{2}{\mu+1}} \sech\left(\sqrt{\mu-1}x+C_2 \right)$.

Now we will see what happens between $x^*$ and $x^{**}$ (in fact we will prove that necessarily $x^*=x^{**}$). If $x^* < x^{**}$ then there exists $\kappa$ such that $[x^*,x^*+\kappa)\subset [x^*,x^{**})$. Given that the term $(\mu u^2-1)$ takes negative values in this interval, the problem of minimizing the energy density has a non-zero solution for the values of $v$. Indeed the problem :
\[
 \min_{v>0} \frac{v^4}{4}+\frac{v^2}{2}(\mu u^2-1),
\]
is solved exactly by $v$ such that $v^2=(1-\mu u^2)$. Therefore, $v^2=(1-\mu u^2)$ for all $x\in[x^*,x^{*}+\kappa)$, this contradicts the definition of $x^{*}$. Finally, given that $u$ is continuous, we get that $u(x^*)=\frac{1}{\sqrt{\mu}}$, which directly implies $v(x^{*-})=v(x^{*+})=0$. Since this is the only possible discontinuity point for $v$, we deduce that $v$ is continuous.

The minimizer found coincides (except for a translation) with $(u_0,v_0)$. Given that the energy functional remains unchanged under translations, we conclude that $(u_0,v_0)$ has the same energy value as $(u,v)$, therefore it is a minimizer.
\end{proof}
\end{lemma}

\section{Behaviour of the solutions as perturbation parameter goes to zero}
\label{LimitBehaviour}
To establish a connection between the solutions of the problem when $\varepsilon=0$ and the case $\varepsilon>0$, we rely on the following result:

\begin{lemma}
\label{chainofinequalities}
We have the following chain of inequalities
\[
\tilde{E}({u}_0,{v}_0)\leq E_\varepsilon(u_\varepsilon,v_\varepsilon)\leq \tilde{E}(u_0,v_0)+O\left(\varepsilon^2{{\ln\frac{1}{\varepsilon}}}\right).
\]
\end{lemma}
As a consequence:
\begin{corollary}
\label{Corollary} For any sequence of $\varepsilon\to 0$, we have that
\[ \int \frac{({v_\varepsilon}')^2}{2}=O\left({\ln\frac{1}{\varepsilon}}\right).\]
\end{corollary}

\begin{proof}[Proof of Lemma \ref{chainofinequalities}]
We divide the proof into upper bound \textbf{UB}, and lower bound \textbf{LB}:
\begin{itemize}
\item[\textbf{LB}] We have $\tilde{E}(u_\varepsilon,v_\varepsilon)\leq E_\varepsilon(u_\varepsilon,v_\varepsilon)$, and since $(u_0,v_0)$ minimizes $\tilde{E}$ in  $X$:
\[
\tilde{E}(u_0,v_0)\leq E_\varepsilon(u_\varepsilon,v_\varepsilon).
\]
\item[\textbf{UB}] For this part, we introduce the function $\bar{v}_0$ defined as $v_0$ with a modification in the interval $I_\varepsilon=(z-\varepsilon^\alpha,z+\varepsilon^\alpha)$, where $z$ is the first value where $v_0$ becomes $0$, and the values in the interval are linearly interpolated between $v_0(z-\varepsilon^\alpha)$ and $v_0(z+\varepsilon^\alpha)$.
Since $(u_\varepsilon,v_\varepsilon)$ minimizes $E$ in $X$ and $(u_0,\bar{v}_0)\in X$ then:
\begin{align}
E(u_\varepsilon,v_\varepsilon)& \leq E_\varepsilon(u_0,\bar{v}_0)\\
&= \tilde{E}(u_0,\bar{v}_0)+\varepsilon^2 \int \frac{(\bar{v}_0')^2}{2}.
\end{align}
To complete the proof we claim that
\[
\varepsilon^2 \int_{I_\varepsilon}  \frac{(\bar{v}_0')^2}{2} = O(\varepsilon^2),
\]
and that
\[
\varepsilon^2 \int_{I_\varepsilon^c} \frac{({v}_0')^2}{2}=O\left(\varepsilon^2\ln\frac{1}{\varepsilon}\right),
\]
which can be shown by considering the expansion for $v_0$. We know
\[
\sech^2(x)=1-\frac{x^2}{2}+\frac{5x^4}{24}+O(x^6),
\]
and so
\[
v_0^2(x)=1-C_1\sech^2\left( C_2(x-z+C_3)\right),
\]
where $C_1=\frac{2\mu}{\mu+1}, C_2=\sqrt{\mu-1}, C_3=\frac{1}{\sqrt{\mu-1}}\arcsech\left( \sqrt{ \frac{\mu+1}{2\mu} } \right)$. Considering the expansion for $v_0^2(z+h)$ (with $h\leq 0$), we obtain
\begin{align}
v_0^2(z+h)&=1-C_1 \sech^2\left(C_2(h+C_3) \right)\\
&=\underbrace{1-C_1 \sech^2(C_2 C_3)}_{v_0^2(z)=0}+C_1\underbrace{\big[\sech^2(C_2C_3)-\sech^2(C_2(h+C_3)) \big]}_{O(h)}\\
&=O(h).
\end{align}
This implies that
\[
v_0(z-\varepsilon^\alpha)=O(\varepsilon^{\alpha/2}),
\]
and therefore,
\begin{align}
\varepsilon^2\int_{I_\varepsilon} \frac{(\bar{v}_0')^2}{2} &\leq \varepsilon^2 \cdot \varepsilon^\alpha \left(\frac{Const\cdot \varepsilon^{\alpha/2}}{2\varepsilon^\alpha} \right)^2\\
&=O(\varepsilon^2),
\end{align}
which gives us the upper bound
\[
E_\varepsilon(u_\varepsilon,v_\varepsilon)\leq \tilde{E}(u_0,\bar{v}_0)+O(\varepsilon^2)+\varepsilon^2 \int_{I_\varepsilon^c}\frac{(v_0')^2}{2}.
\]
\end{itemize}
For the integration term outside of the interval $I_\varepsilon$, we just notice that
\[
(\sech^2)'(x)=-x+O(x^3),
\]
then the expantion for $(v_0^2)'$ becomes:
\begin{align}
(v_0^2)'(z+h)&=-C_1(\sech^2)'(C_2(C_3+h))\\
&=-C_1 (\sech^2)'(C_2 C_3)-C_1\left[(\sech^2)'(C_2(C_3+h))-(\sech^2)'(C_2 C_3) \right]\\
&=k_1+O(h).
\end{align}
Using the identity $v_0'(z+h)=\frac{(v_0^2)'(z+h)}{2v_0(z+h)}$, we obtain
\[
v_0'(z+h)=\frac{k_1+O(h)}{O(h^{1/2})}=O(h^{-1/2})+O(h^{1/2}).
\]
Using the above:
\begin{align}
\varepsilon^2 \int_{-\infty}^{z-\varepsilon^{\alpha}}\frac{(v_0')^2}{2}dx&= \frac{\varepsilon^2}{2} \int^{-\varepsilon^\alpha}_{-\infty}
(v_0')^2(z+h)dh\\
&\leq \frac{\varepsilon^2}{2} \int_{-\infty}^{-1}(v_0')^2(z+h) dh+\frac{\varepsilon^2}{2} \int^{-\varepsilon^\alpha}_{-1}(O(h^{-1/2}))^2 dh\\
&\leq O(\varepsilon^2)+O\left(\varepsilon^2 \ln\left(\frac{1}{\varepsilon}\right)\right)\\
&=O\left(\varepsilon^2 \ln\left(\frac{1}{\varepsilon}\right)\right).
\end{align}

Finally, note that
\[
\tilde{E}(u_0,\bar{v}_0)-\tilde{E}(u_0,v_0)\leq C|I_\varepsilon|,
\]
so by taking $\alpha=2$, we easily obtain the desired estimates:
\[
\tilde{E}({u}_0,{v}_0)\leq E_\varepsilon(u_\varepsilon,v_\varepsilon)\leq \tilde{E}(u_0,v_0)+O\left(\varepsilon^2{{\ln\frac{1}{\varepsilon}}}\right).
\]
\end{proof}
\begin{proof}[Proof of the Corollary \ref{Corollary}]
First notice that
\[
\tilde{E}(u_0,v_0) \leq \tilde{E}_\varepsilon(u_\varepsilon,v_\varepsilon),
\]
then
\[
 \tilde{E}(u_0,v_0)+\varepsilon^2\int \frac{(v_\varepsilon')^2}{2} \leq \tilde{E}_\varepsilon(u_\varepsilon,v_\varepsilon)+\varepsilon^2\int \frac{(v_\varepsilon')^2}{2}.
\]
From Theorem \ref{chainofinequalities} we can have an estimate for the previous and so we get:
\[
 \tilde{E}(u_0,v_0)+\varepsilon^2\int \frac{(v_\varepsilon')^2}{2} \leq \tilde{E}_\varepsilon(u_\varepsilon,v_\varepsilon)+\varepsilon^2\int \frac{(v_\varepsilon')^2}{2}\leq \tilde{E}(u_0,v_0)+O\left(\varepsilon^2{{\ln\frac{1}{\varepsilon}}}\right),
\]
finally, by transitivity we get
\[
\tilde{E}(u_0,v_0)+\varepsilon^2\int \frac{(v_\varepsilon')^2}{2}\leq \tilde{E}(u_0,v_0)+O\left(\varepsilon^2{{\ln\frac{1}{\varepsilon}}}\right),
\]
and so we obtain
\[
 \varepsilon^2\int \frac{(v_\varepsilon')^2}{2}= O\left(\varepsilon^2{{\ln\frac{1}{\varepsilon}}}\right).
\]

\end{proof}

Now we present the main results of the section.
\begin{theorem}
For any sequence of $\varepsilon \to 0$ there exists a sub-sequence such that $(u_\varepsilon,v_\varepsilon) \to (\tilde{u},\tilde{v})$ pointwise in all $\mathbb{R}$, and in $C^{1,\alpha}(-R,R)$ for $u_\varepsilon$ for all $R>0$, $\alpha\in (0,1/2)$, where $(\tilde{u},\tilde{v})$ is the unique solution of
\[
\begin{cases}
\tilde{u}''+&\tilde{u}(1-\tilde{u}^2-\mu \tilde{v}^2)=0,\\
&\tilde{v}(1-\tilde{v}^2-\mu \tilde{u}^2)=0,
\end{cases}
\]
that satisfies $(\tilde{u},\tilde{v})\to (1,0)$ as $x\to +\infty$, and $(\tilde{u},\tilde{v})\to (0,1)$ as $x\to -\infty$, and minimizes $\tilde{E}(\cdot,\cdot)$.
\begin{proof}
We know that $\|u_\varepsilon \|_{L^\infty}\leq 1$ and $\|v_\varepsilon \|_{L^\infty} \leq 1$ for all values of $\varepsilon$. Given that they are solutions of the second order system we get that $\|u''_\varepsilon \|_{L^\infty}$ have uniform bounds for all values of $\varepsilon$, then it is easy to see that
\[
\|u_\varepsilon\|_{H^2(-R,R)} \leq C(R), \, \|v_\varepsilon\|_{L^\infty(-R,R)} \leq C(R).
\]
By Rellich–Kondrachov theorem and Helly's selection theorem, we obtain that there exists a sub-sequence of $\varepsilon$ such that
\[
(u_\varepsilon,v_\varepsilon)\xrightarrow{\varepsilon\to 0} (\tilde{u},\tilde{v}),
\]
where the limit is in $C^{1,\alpha}_{loc}$ for any $\alpha \in (0,1/2)$ for $u_\varepsilon$, and pointwise for $v_\varepsilon$.
Multiplying both equations of the system by some smooth compact support function and using dominated convergence theorem we get that necessarily $(\tilde{u},\tilde{v})$ satisfies
\[
\begin{cases}
\tilde{u}''+&\tilde{u}(1-\tilde{u}^2-\mu \tilde{v}^2)=0\\
&\tilde{v}(1-\tilde{v}^2-\mu \tilde{u}^2)=0.
\end{cases}
\]
Now, let us see that $(\tilde{u},\tilde{v})$ minimizes $\tilde{E}$ and as a consequence it is a translation  of $(u_0,v_0)$.

Given that
\[
\tilde{E}(u,v):=\int_{-\infty}^{\infty} \frac{(u')^2}{2}+\frac14 (u^2+v^2-1)^2+\frac{(\mu-1)}{2}u^2 v^2,
\]
is a non-negative functional we can make use of Fatou's lemma. In this case we obtain
\[
\tilde{E}(\tilde{u},\tilde{v}) \leq \liminf_{\varepsilon \to 0} \tilde{E}(u_\varepsilon, v_\varepsilon).
\]
Let us remember that we also have
\[
 \tilde{E}(u_\varepsilon, v_\varepsilon) \leq  \tilde{E}(u_0, v_0)+O\left(\varepsilon^2{{\ln \frac{1}{\varepsilon}}}\right),
\]
which, together with the previous inequality we conclude
\[
\tilde{E}(\tilde{u},\tilde{v}) \leq \liminf_{\varepsilon \to 0} \tilde{E}(u_0, v_0)+O\left(\varepsilon^2{{\ln \frac{1}{\varepsilon}}}\right),
\]
and so, necessarily $\tilde{E}(\tilde{u},\tilde{v}) \leq \tilde{E}(u_0, v_0)$.

For the proof of the corresponding limits at infinity of $(\tilde{u}, \tilde{v})$, first note that $\tilde{u}$, being the pointwise limit of a sequence of nonnegative, increasing functions, is also a nonnegative and increasing function. The same applies to $\tilde{v}$, which is the limit of a sequence of nonnegative, decreasing functions, and therefore is nonnegative and decreasing. Additionally, observe that the integral of $\tilde{E}(\tilde{u}, \tilde{v})$ is finite. Consequently, at infinity, the pair $(\tilde{u}, \tilde{v})$ must have either $(1, 0)$ or $(0, 1)$ as an accumulation point.

Suppose for the sake of contradiction that $\tilde{u}(x) \to L \neq 1$ as $x \to \infty$. The only remaining possibility is $L = 0$. Since $\tilde{u}$ is increasing and nonnegative, this would imply that $\tilde{u} = 0$, and thus $\tilde{v} = 1$. From the construction of $(u_\varepsilon, v_\varepsilon)$ in the proof of Theorem \ref{existenceproof}, we know that $u_\varepsilon(0) = v_\varepsilon(0)$ for all $\varepsilon$. Therefore, $\tilde{u}(0) = \tilde{v}(0)$, but we just proved that they are different constant functions, which leads to a contradiction. The same contradiction arises if we assume $\tilde{u}(x) \to 1$ as $x \to -\infty$, by monotonicity, this would imply that $\tilde{u} = 1$ and $\tilde{v} = 0$.

With the same argument, we obtain the corresponding limits for $\tilde{v}$. We conclude that $\tilde{E}(\tilde{u},\tilde{v}) = \tilde{E}(u_0, v_0)$, and thus, by Lemma \ref{u0v0teo}, $(\tilde{u}, \tilde{v})$ must be a displacement of $(u_0, v_0)$.
\end{proof}
\end{theorem}
\section{The Painlev\'e equation as a singular limit}
\label{PainleveProblem}
Now we already have solutions for the system for all values of $\varepsilon$
\[
\begin{aligned}
u''+u(1-u^2-\mu v^2)&=0,\\
\varepsilon^2 v'' +v(1-v^2-\mu u^2)&=0.
\end{aligned}
\]
Let us consider the change of variables:
\[
\phi_\varepsilon(t)=\varepsilon^{-1/3} v_\varepsilon(z_\varepsilon+\varepsilon^{2/3}t),
\]
where, given the monotonicity of $u_\varepsilon$ \cite{Stan},  $z_\varepsilon$ is the unique value  (dependent of $\varepsilon$) such that $u_\varepsilon(z_\varepsilon)=\frac{1}{\sqrt{\mu}}$.

The second equation of the system after some algebraic manipulations becomes
\[
 \phi''_\varepsilon- \phi_\varepsilon( \phi^2_\varepsilon+\mu (u_\varepsilon^2)'(z_\varepsilon)t)= \varepsilon^{-2/3} \phi_\varepsilon h_\varepsilon,
\]
where $h_\varepsilon$ is given by:
\[
h_\varepsilon=\left(-\mu\cdot \left((u^2_\varepsilon(z_\varepsilon+\varepsilon^{2/3} t)-u^2_\varepsilon(x))-  (u_\varepsilon^2)'(z_\varepsilon) \varepsilon^{2/3}t \right) \right).
\]

We know that for all values of $\varepsilon$, the function $u_\varepsilon$ is differentiable and has smooth derivative, we have
\[
u^2_\varepsilon(z_\varepsilon+\varepsilon^{2/3} t)-u_\varepsilon^2(z_\varepsilon)=(u_\varepsilon^2)'(z_\varepsilon)\varepsilon^{2/3}t+o(\varepsilon^{2/3} t),
\]
hence
\[
\phi''_\varepsilon- \phi_\varepsilon( \phi^2_\varepsilon+\mu (u_\varepsilon^2)'(z_\varepsilon)t)=-\mu \varepsilon^{-2/3} \phi_\varepsilon \cdot o(\varepsilon^{2/3} t).
\]
\begin{lemma}
For any value of $R>0$, the family of functions $\{\phi_\varepsilon \}_\varepsilon$ is  bounded  in $L^\infty(-R,R)$.
\begin{proof} Using the arguments presented in Lemma 3.1 from \cite{shadowkink}, we are able to obtain local uniform estimates for $\phi_\varepsilon$. In their work, the authors derive a uniform estimate for solutions of a similar equation which, when taking the constant $a=0$, becomes:
\[
\varepsilon^2 v'' + v(\tilde{\mu}(x) - v^2) = 0,
\]
where $\tilde{\mu}$ is an even bump function with zeros at $\pm \xi$. In our case, if we consider $\tilde{\mu}(x) := 1 - \mu u_\varepsilon^2$, we can employ the same techniques applied in their work (with simplifications, since our function $\tilde{\mu}$ is decreasing and has only one zero, so only one of the two cases they treat must be considered) to obtain the desired estimate:
\[
v(z_\varepsilon + t\varepsilon^{2/3}) \leq \sqrt{8\lambda \left[1 - \mu u_\varepsilon(z_\varepsilon + t\varepsilon^{2/3}) \right]}, \quad \text{for all } t\in (-\infty, +\infty),
\]
where $\lambda > 1$ is some real number. From this, we directly obtain that:
\[
\frac{v(z_\varepsilon + t\varepsilon^{2/3})}{\varepsilon^{1/3}} = o(1),
\]
and so the proof is finished.
\end{proof}
\end{lemma}
\begin{lemma}
For any value of $R>0$ we have that $h_\varepsilon|_{(-R,R)}$ is $o(\varepsilon^{2/3})$.
\begin{proof}
Clearly, given that $u_\varepsilon$ is differentiable on $\mathbb{R}$, we have that $u_\varepsilon^2$ also is differentiable on $\mathbb{R}$, and so
\[
(u^2_\varepsilon(z_\varepsilon+\varepsilon^{2/3} t)-u^2_\varepsilon(z_\varepsilon))-  (u_\varepsilon^2)'(z_\varepsilon) \varepsilon^{2/3}t = o(\varepsilon^{2/3}t).
\]
Then, $h_\varepsilon(t)=o(\varepsilon^{2/3} t)$. Now if $t$ is bounded, the result is trivial.
\end{proof}
\end{lemma}
\begin{theorem} For any sequence of $\varepsilon \to 0$, there exists a subsequence such that $\phi_\varepsilon \to {\phi}_0$, pointwise in $\mathbb{R}$. This limit function is a solution of the problem
\begin{equation}
\label{eq:painleve}
{\phi}''-{\phi}({\phi}^2+\mu (u_0^2)'(z) t)=0,\quad  \forall t.
\end{equation}
\begin{proof}
Let us consider some $R>0$. We have proved that $\|\phi_\varepsilon\|_{L^\infty(-R,R)}$ is uniformly bounded in $\varepsilon$.

By Helly's selection theorem, there exists a convergent subsequence. Thus, we can construct a pointwise convergent subsequence to some $\phi_0$ on the whole real line.

Now, $\phi_\varepsilon$ is a solution of the equation
\[
 \phi''_\varepsilon- \phi_\varepsilon( \phi^2_\varepsilon+\mu (u_\varepsilon^2)'(z_\varepsilon)t)= \varepsilon^{-2/3} \phi_\varepsilon {h_\varepsilon}.
\]
If we multiply both sides of the previous equation by a compactly supported function $\varphi$ and integrate, we obtain
\[
 \int \phi_\varepsilon \varphi'' - \phi^3_\varepsilon \varphi+\mu (u_\varepsilon^2)'(z_\varepsilon)\varphi t dt= \int \varepsilon^{-2/3} \phi_\varepsilon h_\varepsilon \varphi.
\]
Given that $\varphi$ has compact support, say $K \subset \mathbb{R}$, we have that $(\phi_\varepsilon)|_{K}\leq C_K$. By dominated convergence theorem, the left-hand side of the above equation converges along a subsequence since the sequence of $(\phi_\varepsilon)_\varepsilon$ does. For the right-hand side, we have already shown $\frac{h_\varepsilon}{\varepsilon^{2/3}}\to 0$ as $\varepsilon \to 0$. It follows that $\varepsilon^{-2/3} \phi_\varepsilon h_\varepsilon\to 0$ as $\varepsilon \to 0$. Thus, by the dominated convergence theorem, we conclude that $\phi_0$ solves in the weak sense the equation (\ref{eq:painleve}).
\end{proof}
\end{theorem}

\medskip
\medskip


\begin{thebibliography}{99}

\bibitem{alama1997}
\newblock S. Alama, L. Bronsard, C. Gui,
\newblock Stationary layered solutions in $\mathbb{R}^2$ for an Allen-Cahn system with multiple well potential
\newblock , {Calculus of Variations and P.D.E.}, \textbf{5}, (1997), 359-390.

\bibitem{Stan}
\newblock S. Alama, L. Bronsard, A. Contreras, and D. Pelinovsky,
\newblock Domain Walls in the Coupled Gross-Pitaevskii Equations,
\newblock {Archive For Rational Mechanics And Analysis}. \textbf{215}, (2015), 579--610.

\bibitem{Alikakos2}
\newblock N. Alikakos, P. Fife, G. Fusco, and C. Sourdis,
\newblock Singular Perturbation Problems Arising from the Anisotropy of Crystalline Grain Boundaries,
\newblock {Journal Of Dynamics And Differential Equations}. \textbf{19}, (2007), 935-949.

\bibitem{Alikakos3}
\newblock N. Alikakos, P. Fife, G. Fusco, and C. Sourdis,
\newblock {Analysis of the heteroclinic connection in a singularly perturbed system arising from the study of crystalline grain boundaries},
\newblock {Interfaces And Free Boundaries - INTERFACE FREE BOUND}, \textbf{8} (2006), 159-183.

\bibitem{DynTopMag}
\newblock V. G. Bar’yakhtar, M. V. Chetkin, B. A. Ivanov, and
S. N. Gadetskii,
\newblock {Dynamics of Topological Magnetic Solitons},
\newblock (Springer-Verlag, Berlin, 1994).

\bibitem{shadowkink}
\newblock M.G. Clerc, J.D Dávila, M. Kowalczyk, P. Smyrnelis and E. Vidal-Henriquez,
\newblock {Theory of light-matter interaction in nematic liquid crystals and the second Painlevé equation}, Report of Hong Kong SARS Expert Committee,
\newblock {Calc. Var.} \textbf{56}, 93 (2017).

\bibitem{Contreras2018}
\newblock A. Contreras, D.E. Pelinovsky, and M. Plum,
\newblock {Orbital Stability of Domain Walls in Coupled Gross--Pitaevskii Systems},
\newblock In \emph{SIAM Journal on Mathematical Analysis}, Vol. \textbf{50}, Issue 1, (2018) 810–833.

\bibitem{Contreras2022}
\newblock A. Contreras, D.E. Pelinovsky, and V. Slastikov,
\newblock {Domain walls in the coupled Gross–Pitaevskii equations with the harmonic potential},
\newblock In Calculus of Variations and Partial Differential Equations (Vol. 61, Issue 5). Springer Science and Business Media LLC (2022). https://doi.org/10.1007/s00526-022-02277-6.

\bibitem{Cross}
\newblock M. Cross, H. Greenside,
\newblock Pattern Formation and Dynamics in Nonequilibrium Systems,
\newblock Cambridge University Press; 2009.

\bibitem{Ferroelectric}
\newblock D. Damjanovic,Rep. Prog. Phys. 61, 1267 (1998).

\bibitem{ThePhysicsOfLiquidCrystals}
\newblock P. G. de Gennes and J. Prost. The Physics of Liquid Crystals
(Oxford University Press, New York, 1995).

\bibitem{Jaksch2003}
\newblock D. Jaksch (2003). Bose-Einstein Condensation. In Journal of Physics A: Mathematical and General (Vol. 36, Issue 37, pp. 9797). IOP Publishing. https://doi.org/10.1088/0305-4470/36/37/701.

\bibitem{Malomed}
\newblock B.A. Malomed,
\newblock {New findings for the old problem: Exact solutions for domain walls in coupled real Ginzburg-Landau equations},
\newblock {Physics Letters A}. \textbf{Vol 422}, (2022) https://doi.org/10.1016/j.physleta.2021.127802.

\bibitem{Nirnberg}
\newblock L. Nirenberg,
\newblock {On elliptic partial differential equations.},
\newblock {Annali Della Scuola Normale Superiore Di Pisa - Scienze Fisiche E Matematiche}. \textbf{Ser. 3, 13}, 115-162 (1959).

\bibitem{Pismen}
\newblock L.M. Pismen (2009), Patterns and Interfaces in Dissipative Dynamics. In Encyclopedia of Complexity and Systems Science (pp. 6459–6476). Springer New York. https://doi.org/10.1007/978-0-387-30440-3\_381.

\bibitem{Sourdis}
\newblock C. Sourdis, and P. Fife,
\newblock {Existence of heteroclinic orbits for a corner layer problem in anisotropic interfaces.},
\newblock {Advances In Differential Equations}. \textbf{12} (2007,1).
\end{thebibliography}
\end{document}